\documentclass[12pt,leqno,oneside]{amsart}
\usepackage{mathrsfs,dsfont}
\usepackage{amsmath,amstext,amsthm,amssymb,bbm}
\usepackage{charter}
\usepackage{typearea}
\usepackage{hyperref}
\usepackage{color}

\usepackage{float}
\usepackage{tikz}
\usetikzlibrary{patterns}
\usepackage{graphicx}
\usepackage{hyperref}

\usepackage[width=6.4in,height=8.5in]
{geometry}

\def\bt{\begin{thm}}
	\def\et{\end{thm}}
\def\bl{\begin{lem}}
	\def\el{\end{lem}}
\def\bd{\begin{defn}}
	\def\ed{\end{defn}}
\def\bc{\begin{cor}}
	\def\ec{\end{cor}}
\def\bp{\begin{proof}}
	\def\ep{\end{proof}}
\def\br{\begin{rem}}
	\def\er{\end{rem}}

\newtheorem{thm}{Theorem}[section]
\newtheorem{prop}[thm]{Proposition}
\newtheorem{lem}[thm]{Lemma}
\newtheorem{defn}[thm]{Definition}
\newtheorem{example}[thm]{Example}
\newtheorem{rem}[thm]{Remark}
\newtheorem{cor}[thm]{Corollary}

\numberwithin{equation}{section}

\title{Non-null ruled surfaces and pseudo hyper-dual tangent sphere bundles }

\author{ Emel Karaca}
\address{Polatl{\i} Faculty of Science and Arts, Ankara Hac{\i} Bayram Veli University, Ankara, 06900 Turkey}
\email{emel.karaca@hbv.edu.tr}

\author{Senem ÖZTÜRK}
\address{Polatl{\i} Faculty of Science and Arts, Ankara Hac{\i} Bayram Veli University, Ankara, 06900 Turkey}
\email{senemm2000.06@gmail.com}


\date{\today}
\newcommand\blfootnote[1]{ \begingroup \renewcommand\thefootnote{} \footnotetext{#1} \addtocounter{footnote}{-1} \endgroup}
\begin{document}
\blfootnote{This work was supported by TÜBİTAK under Grant No. 125F475 (ARDEB 3501 Career Development Program).}	
\begin{abstract}
In this paper, the tangent bundles of hyper-dual pseudo spheres are defined. Then, a new family of non-null ruled surfaces associated with hyper-dual Lorentzian curves is defined via the E.Study hyper-map in Minkowski space and the definitions of the tangent bundles of hyper-dual pseudo spheres. Within this framework, the resulting non-null ruled surfaces are classified. Necessary and sufficient conditions for developability are obtained with their geometric interpretations in terms of hyper-dual Lorentzian curves and arc-length parameters. Moreover, explicit examples  are presented to verify the theoretical results.  
	\end{abstract}

	\maketitle

	\section{Introduction}

Ruled surfaces constitute one of the fundamental classes of surfaces in differential geometry owing to their intricate geometric structure and numerous applications in kinematics, computer-aided geometric design, robotics, and mathematical physics, etc. Beyond their classical importance, ruled surfaces have become an active research topic in Lorentzian geometry in \cite{lopez}.

The geometric interpretation of ruled surfaces is closely related to the classical E. Study mapping, which constructs a one-to-one relation between directed lines and points on dual unit spheres. This mapping provides an elegant geometric interpretation of ruled surfaces in terms of spherical curves and has become one of the principal tools in line geometry. The pioneering works of Study in \cite{study}, Blaschke in \cite{blas}, Struik in \cite{struik}, and O'Neill in \cite{neil} established the theoretical foundations of this correspondence and its applications to Euclidean, Lorentzian, and semi-Riemannian geometries.
Recent developments in differential geometry have also strengthened this connection through the study of tangent bundles and tangent sphere bundles. The geometry of tangent bundles over unit spheres and pseudo-spheres has revealed important relationships between spherical curves, natural lift curves, and ruled surfaces. In particular, the authors \cite{bekar}, \cite{hathout}, \cite{gen} established several fundamental results concerning tangent sphere bundles in Euclidean and Lorentzian geometries. 

Motivated by these studies, further investigations have focused on various classes of ruled surfaces and their associated tangent bundle structures. Karaca and Çalışkan in \cite{karaca, karaca2, caliskan, kar2} studied the geometry of slant ruled surfaces and their natural lift curves, revealing new geometric characterizations within tangent bundle settings. From a complementary viewpoint, the classical dual geometric formulation of ruled surfaces was further developed by Karakaş, Gündoğan, and co-authors in \cite{karakas}, while Uğurlu and Çalışkan \cite{ug} established the E. Study correspondence for dual Lorentzian and hyperbolic unit spheres. These contributions considerably expanded the applicability of E. Study geometry to Lorentzian line geometry and provide a natural foundation for subsequent generalizations based on higher-order dual structures.
The emergence of hyper-dual numbers has considerably broadened the scope of classical dual geometry by providing a higher-order algebraic framework capable of representing more sophisticated geometric and kinematic structures. Originally introduced as an extension of dual numbers, hyper-dual numbers have found important applications in automatic differentiation, rigid body kinematics, and computational mechanics \cite{cockle,cohen,cohen2,cohen3,fike,fike2}. Their algebraic properties have subsequently motivated the development of hyper-dual vectors, hyper-dual quaternions, and hyper-dual geometric models suitable for studying differential geometric objects in higher-order settings in \cite{bekar2}.

Moreover, in \cite{aslan} and \cite{aslan2}, the authors  investigated hyper-dual curves and the ruled surfaces generated by them, demonstrating that hyper-dual structures naturally extend the classical dual geometric framework. More recently, Derkaoui, Hathout, Bekar, and Yaylı \cite{bekar} established an important relationship between ruled surfaces and hyper-dual tangent sphere bundles, thereby providing a new perspective on the interaction between tangent bundle geometry and higher-order dual structures. Related studies by Karaca \cite{kar} further explored hyper-dual curves, their associated striction curves, and ruled surfaces. Collectively, these studies demonstrate the effectiveness of hyper-dual geometry in extending classical differential geometric constructions.
Despite these significant developments, a comprehensive geometric framework connecting pseudo hyper-dual tangent sphere bundles with the theory of non-null ruled surfaces in the hyper-dual Lorentzian setting has remained incomplete. In particular, the characterization of non-null ruled surfaces through the E. Study hyper-map, their systematic classification according to causal character, and the geometric interpretation of their developability conditions within the pseudo hyper-dual framework have not been investigated in a unified manner. Establishing such a framework is important not only for extending the classical E. Study correspondence but also for providing new geometric tools for the analysis of ruled surfaces associated with higher-order dual structures.

Motivated by this observation, the present paper develops a unified geometric formulation for non-null ruled surfaces associated with pseudo hyper-dual tangent sphere bundles. First, an E.~Study correspondence is established for unit hyper-dual vectors \cite{gok}, leading to a natural construction of non-null ruled surfaces generated by hyper-dual Lorentzian curves. Subsequently, these ruled surfaces are classified, and necessary and sufficient developability conditions together with their geometric interpretations are obtained. Furthermore, several structural properties of the resulting ruled surfaces, common base curves, and orthogonal ruling directions, are investigated. The theoretical results are finally illustrated by explicit examples, confirming the proposed geometric framework. \\

The paper is structured as follows. Section~2 provides the fundamental background on dual numbers, Lorentzian geometry, and the tangent bundles of pseudo-spheres. Section 3 introduces hyper-dual numbers and hyper-dual tangent sphere bundles, and establishes the E.~Study map for unit hyper-dual vectors. In Section~4, the correspondence between pseudo hyper-dual tangent sphere bundles and non-null ruled surfaces is introduced. Within this setting, non-null ruled surfaces are classified, developability conditions are derived and interpreted geometrically, and several structural properties are established. Explicit examples are also presented to illustrate the theoretical results. Finally, Section~5 summarizes the main contributions and discusses the significance of the proposed framework. 

\section{Dual Numbers and tangent bundle of pseudo-sphere}
This section introduces the tangent bundles of pseudo-spheres and explores their relationship with dual numbers. \\ Let \(\mathbb{R}_{1}^{3}\) be the three-dimensional Minkowski space equipped with the Lorentzian inner product defined by
\begin{equation}
\langle , \rangle_{L}=x_{1}y_{1}+x_{2}y_{2}-x_{3}y_{3}
\end{equation}
where $\vec{x}=(x_{1},x_{2},x_{3})$ and $\vec{y}=(y_{1},y_{2},y_{3})$ denote the vectors in $\mathbb{R}_{1}^{3}$. In \(\mathbb{R}_{1}^{3}\), the vector \(\vec{x}\) is characterized as follows: \\
$(i)$ Spacelike, if \(\langle \vec{x},\vec{x}\rangle_{L}>0\) or \(\vec{x}=0\), \\
$(ii)$ Timelike, if \(\langle \vec{x},\vec{x}\rangle_{L}<0\), \\
$(iii)$ Lightlike, if \(\langle \vec{x},\vec{x}\rangle_{L}=0\) and \(\vec{x}\neq 0\). \\
The Lorentzian norm of \(\vec{x}\) is defined by
\[
||\vec{x}||_{L}=\sqrt{|\langle\vec{x},\vec{x}\rangle_{L}|}.
\]
Two vectors \(\vec{x}\) and \(\vec{y}\) are said to be Lorentzian orthogonal whenever \(\langle\vec{x},\vec{y}\rangle_{L}=0\). The Lorentzian vector product of these vectors is expressed as follows:
\begin{equation}
\vec{x} \times_{L} \vec{y}=(x_{2}y_{3}-x_{3}y_{2},x_{3}y_{1}-x_{1}y_{3}, x_{2}y_{1}-x_{1}y_{2} ),
\end{equation}
where $\vec{x}=(x_{1},x_{2},x_{3})$ and $\vec{y}=(y_{1},y_{2},y_{3})$ in $\mathbb{R}_{1}^{3}$, \cite{neil}.
Lorentzian unit sphere and hyperbolic unit sphere are given as follows:
\begin{eqnarray}
\mathbb{S}_{1}^{2}&=&\{\vec{x}=(x_{1},x_{2},x_{3}) \in \mathbb{E}_{1}^{3}: \langle \vec{x},\vec{x} \rangle_{L}=1 \}, \\ 
\mathbb{H}_{0}^{2}&=&\{\vec{x}=(x_{1},x_{2},x_{3}) \in \mathbb{E}_{1}^{3}: \langle \vec{x}, \vec{x} \rangle_{L}=-1\},
\end{eqnarray}
see \cite{neil}. The tangent bundles of $\mathbb{S}_{1}^{2}$ and $\mathbb{H}_{0}^{2}$ are 
\begin{eqnarray}
T\mathbb{S}_{1}^{2}&=&\{(\gamma, \vartheta) \in \mathbb{S}_{1}^{2} \times \mathbb{E}_{1}^{3}: \langle \gamma, \vartheta \rangle_{L}=0 \}, \\ 
T\mathbb{H}_{0}^{2}&=&\{(\gamma, \vartheta) \in \mathbb{H}_{0}^{2} \times \mathbb{E}_{1}^{3}: \langle \gamma, \vartheta \rangle_{L}=0\}.
\end{eqnarray}
The unit tangent bundles of $\mathbb{S}_{1}^{2}$ and $\mathbb{H}_{0}^{2}$ are represented as follows
\begin{eqnarray}
UT\mathbb{S}_{1}^{2}&=&\{(\gamma, \vartheta) \in \mathbb{S}_{1}^{2} \times \mathbb{E}_{1}^{3}: \langle \gamma, \vartheta \rangle_{L}=0, \langle \vartheta, \vartheta \rangle_{L}=\pm 1 \}, \\ 
UT\mathbb{H}_{0}^{2}&=&\{(\gamma, \vartheta) \in \mathbb{H}_{0}^{2} \times \mathbb{E}_{1}^{3}: \langle \gamma, \vartheta \rangle_{L}=0, \langle \vartheta, \vartheta \rangle_{L}= 1\}.
\end{eqnarray}
Let \(I,J\subset\mathbb{R}\) denote two open intervals. Let $\beta$ and $\gamma$ be the base curve and vector field along $\beta$. The ruled surface in $\mathbb{R}_{1}^{3}$ is 
\begin{equation*}
\phi(t,v)=\beta(t)+v\gamma(t), \ \ t \in I \subset \mathbb{R}, \ v \in J \subset \mathbb{R}.
\end{equation*} 
Moreover, $\phi$ is developable iff the following relation holds:
\begin{equation*}
det(\beta^{'}(t), \gamma(t),\gamma^{'}(t))=0.
\end{equation*}
\subsection{The tangent bundle of pseudo-sphere and dual numbers}
The set of dual numbers is defined as
$$D=\{X=x+\varepsilon x^{*}: (x, x^{*}) \in \mathbb{R}\times \mathbb{R}, \ \ \varepsilon^{2}=0\}.$$
Here, $x$ and $x^{*}$ denote real and dual parts of dual number, respectively, see \cite{blas}.
In $\mathbb{R}^{3}$, the combination of $\vec{x}$ and $\vec{x}^{*}$ is referred to as dual vectors. $\vec{X}=\vec{x}+\varepsilon \vec{x}^{*}$ is defined as a dual vector if $\vec{x}$ and $\vec{x}^{*}$ are vectors in $\mathbb{R}^{3}$. The set of dual vectors, $\mathbb{D}^{3},$ is expressed as
\begin{equation*}
\mathbb{D}^{3}=\{ \vec{X}=(X_{1},X_{2},X_{3}):X_{i}=x_{i}+\epsilon x_{i}^{*} \in \mathbb D, 1 \leq i \leq 3 \}.
\end{equation*}
Consider two dual vectors \(\vec{X}=\vec{x}+\varepsilon\vec{x}^{*}\) and \(\vec{Y}=\vec{y}+\varepsilon\vec{y}^{*}\). Their sum, Lorentzian inner product, and Lorentzian vector product are defined, respectively, by
\begin{eqnarray*}
\vec{X}+_{L}\vec{Y}&=& (\vec{x}+\vec{y})+\varepsilon (\vec{x}^{*}+\vec{y}^{*}), \\
\langle \vec{X},\vec{Y} \rangle_{L}&=& \langle \vec{x},\vec{y} \rangle_{L}+ \varepsilon (\langle \vec{x}^{*},\vec{y}\rangle_{L}+\langle \vec{x},\vec{y}^{*}\rangle_{L}),
\end{eqnarray*}
and
$$\vec{X} \times_{L} \vec{Y}= \vec{x}\times_{L} \vec{y}+ \varepsilon(\vec{x} \times_{L} \vec{y}^{*}+\vec{x}^{*}\times_{L} \vec{y}).$$
The Lorentzian norm of $\vec{X}=\vec{x}+\varepsilon \vec{x}^{*}$ is 
\begin{equation}
	\vert \vert \vec{X}\vert \vert_{L}= \sqrt{\vert \langle \vec{x}, \vec{x}\rangle_{L} \vert}+\varepsilon \frac{sgn(\langle \vec{x},\vec{x} \rangle_{L}) \langle \vec{x},\vec{x}^{*}\rangle_{L}}{\sqrt{ \vert \langle \vec{x},\vec{x}\rangle_{L} \vert}}, \ \ \langle \vec{x}, \vec{x}\rangle_{L} \neq 0.
\end{equation}
If $\vert \vert \vec{X}\vert \vert_{L}=1$, $\vec{X}$ is referred as the unit Lorentzian dual vector. Moreover, the dual vector \(\vec{X}=\vec{x}+\varepsilon\vec{x}^{*}\) inherits its causal character from \(\vec{x}\), being spacelike, timelike, or lightlike according to the causal type of its real component, see \cite{lopez}. The space of Lorentzian dual vectors is represented by \(\mathbb{D}_{1}^{3}\). \\ 
The dual hyperbolic unit sphere is defined as the set of all unit dual timelike vectors:
\begin{equation}
H^{2}=\{ \vec{X}=\vec{x}+\varepsilon \vec{x}^{*} \in \mathbb{D}_{1}^{3}: \langle \vec{x},\vec{x}\rangle_{L}=-1, \langle \vec{x},\vec{x}^{*}\rangle_{L}=0 \}
\end{equation}
and the dual Lorentzian unit sphere is defined as the set of all unit dual spacelike vectors:
\begin{equation}
S_{1}^{2}=\{ \vec{X}=\vec{x}+\varepsilon \vec{x}^{*} \in \mathbb{D}_{1}^{3}: \langle \vec{x},\vec{x}\rangle_{L}=1, \langle \vec{x},\vec{x}^{*}\rangle_{L}=0 \}.
\end{equation}
Now, we can give the following isomorphisms:
\begin{eqnarray}\label{212}
H^{2} &\longrightarrow& T\mathbb{H}_{0}^{2}, \\
\vec{X}=\vec{x}+\varepsilon \vec{x}^{*} &\longmapsto& \Gamma(t)=(\vec{x}(t),\vec{x}^{*}(t)) \nonumber
\end{eqnarray}
and 
\begin{eqnarray}\label{213}
S_{1}^{2} &\longrightarrow& T\mathbb{S}_{1}^{2}, \\
	\vec{X}=\vec{x}+\varepsilon \vec{x}^{*} &\longmapsto& \Gamma(t)=(\vec{x}(t),\vec{x}^{*}(t)). \nonumber
\end{eqnarray}
The following theorem provides a geometric interpretation of the above construction.
\begin{thm}[E. Study Map, \cite{ug} ] 
A bijective correspondence exists between the oriented timelike (resp., spacelike) lines in \(\mathbb{R}_{1}^{3}\) and points of the dual hyperbolic unit sphere (resp., dual Lorentzian unit sphere).
\end{thm}
A ruled surface in \(\mathbb{R}_{1}^{3}\) is the locus traced out by a moving straight line. Such a surface is called spacelike if the induced metric is positive definite (equivalently, its normal vector is timelike everywhere), whereas it is called timelike if the induced metric has Lorentzian signature (equivalently, its normal vector is spacelike everywhere).
\section{The tangent bundles of $\mathbb{S}_{1}^{2}$ and $\mathbb{H}_{0}^{2}$ and hyper-dual numbers}
In this section, some concepts of hyper-dual numbers and the tangent bundles of $\mathbb{S}_{1}^{2}$ and $\mathbb{H}_{0}^{2}$ are introduced. \\ 
The set of hyper-dual numbers is 
\begin{eqnarray*}
D_{2}&=&\{ X_{2}=(x_{0}+\varepsilon x_{1})+\varepsilon^{*}(x_{2}+\varepsilon x_{3}): x_{0},x_{1},x_{2},x_{3} \in \mathbb{R}\} \\
&=& \{ X_{2}=X+\varepsilon^{*}X^{*}: X, X^{*}\in D\},
\end{eqnarray*}
where $\varepsilon^{2}=\varepsilon^{*2}=(\varepsilon\varepsilon^{*})^{2}=0$ and $\varepsilon \neq 0, \varepsilon \neq \varepsilon^{*}, \varepsilon\varepsilon^{*}=\varepsilon^{*}\varepsilon \neq 0$ in \cite{fike}. The operations of addition and inner product on any pair of hyper-dual numbers are
\begin{eqnarray*}
X_{2}&=&X+\varepsilon^{*}X^{*}=x_{0}+\varepsilon x_{1}+\varepsilon^{*}x_{2}+\varepsilon \varepsilon^{*} x_{3}, \ \  \\
Y_{2}&=&Y+\varepsilon^{*}Y^{*}=y_{0}+\varepsilon y_{1}+\varepsilon^{*}y_{2}+\varepsilon \varepsilon^{*} y_{3}
\end{eqnarray*}
are, respectively, defined as
\begin{eqnarray*}
X_{2}+Y_{2}&=&(X+Y)+\varepsilon^{*}(X^{*}+Y^{*}) \\
&=&(x_{0}+y_{0})+\varepsilon(x_{1}+y_{1})+\varepsilon^{*}(x_{2}+y_{2})+\varepsilon\varepsilon^{*}(x_{3}+y_{3})
\end{eqnarray*}
and 
\begin{eqnarray*}
\langle X_{2},Y_{2} \rangle&=&\langle X,Y \rangle+\varepsilon^{*}(\langle X^{*},Y \rangle+\langle X, Y^{*} \rangle) \\
	&=&(\langle x_{0}, y_{0} \rangle)+\varepsilon(\langle x_{0}, y_{1} \rangle+\langle x_{1}, y_{0} \rangle)+\varepsilon^{*}(\langle x_{0}, y_{2} \rangle+ \langle x_{2}, y_{0} \rangle) \\
    &+&\varepsilon\varepsilon^{*}(\langle x_{0},y_{3} \rangle+ \langle x_{1},y_{2} \rangle+ \langle x_{2}, y_{1} \rangle+ \langle x_{3}, y_{0} \rangle).
\end{eqnarray*}
The set of all hyper-dual vectors is 
\begin{equation*}
\hat{\mathbb{D}}_{2}=\{ \vec{X}_{2}=\vec{X}+\varepsilon^{*} \vec{X}^{*}: \vec{X}, \vec{X}^{*} \in \mathbb{D}^{3}: \varepsilon^{*2}=0 \}.
\end{equation*}
Let $\vec{X}_{2}=\vec{X}+\varepsilon^{*} \vec{X}^{*},$ $\vec{Y}_{2}=\vec{Y}+\varepsilon^{*} \vec{Y}^{*} \in \hat{\mathbb{D}}_{2}.$
The Lorentzian inner and vector products of $\vec{X}_{2}$ and $\vec{Y}_{2}$ are, respectively, given as follows: 
\begin{eqnarray*}
\langle \vec{X}_{2}, \vec{Y}_{2} \rangle_{L}&=& \langle \vec{X}, \vec{Y}\rangle_{L}+\varepsilon^{*}(\langle \vec{X}, \vec{Y}^{*} \rangle_{L}+ \langle \vec{X}^{*}, \vec{Y} \rangle_{L}) \\
&=& \langle \vec{x}_{0}, \vec{y}_{0} \rangle_{L}+\varepsilon(\langle \vec{x}_{0}, \vec{y}_{1} \rangle_{L}+\langle \vec{x}_{1}, \vec{y}_{0} \rangle_{L})+\varepsilon^{*}(\langle \vec{x}_{0}, \vec{y}_{2} \rangle_{L}+\langle \vec{x}_{2}, \vec{y}_{0} \rangle_{L}) \\
&+& \varepsilon \varepsilon^{*}(\langle \vec{x}_{0}, \vec{y}_{3} \rangle_{L}+\langle \vec{x}_{1}, \vec{y}_{2} \rangle_{L}+\langle \vec{x}_{2}, \vec{y}_{1} \rangle_{L}+\langle \vec{x}_{3}, \vec{y}_{0} \rangle_{L})
\end{eqnarray*}
and 
\begin{eqnarray*}
\vec{X}_{2} \times_{L} \vec{Y}_{2} &=&  \vec{X} \times_{L} \vec{Y}+\varepsilon^{*}( \vec{X} \times_{L} \vec{Y}^{*} +  \vec{X}^{*} \times_{L} \vec{Y}) \\
	&=&  \vec{x}_{0} \times_{L} \vec{y}_{0} +\varepsilon( \vec{x}_{0} \times_{L} \vec{y}_{1} + \vec{x}_{1} \times_{L} \vec{y}_{0} )+\varepsilon^{*}( \vec{x}_{0} \times_{L} \vec{y}_{2} + \vec{x}_{2} \times_{L} \vec{y}_{0} ) \\
	&+& \varepsilon \varepsilon^{*} ( \vec{x}_{0} \times_{L} \vec{y}_{3} + \vec{x}_{1} \times_{L} \vec{y}_{2} +\vec{x}_{2} \times_{L} \vec{y}_{1} + \vec{x}_{3} \times_{L} \vec{y}_{0} ).
\end{eqnarray*}
Moreover, the set 
\begin{equation*}
\hat{\mathbb{D}}^{3}_{1}=\{ \vec{X}_{2}=\vec{X}+\varepsilon^{*} \vec{X}^{*}: \vec{X}, \vec{X}^{*} \in \mathbb{D}_{1}^{3}: \varepsilon^{*2}=0 \}
\end{equation*}
is called a hyper-dual Lorentzian space, see \cite{bekar2}. Every element of this space is termed a hyper-dual split vector. The causal character of \(\vec{X}_{2}=\vec{X}+\varepsilon^{*}\vec{X}^{*}\) is determined by that of \(\vec{X}\), namely, spacelike, timelike, or lightlike. The Lorentzian vector product of two hyper-dual split vectors $\vec{X}_{2}=\vec{X}+\varepsilon^{*} \vec{X}^{*}$ and $\vec{Y}_{2}=\vec{Y}+\varepsilon^{*} \vec{Y}^{*}$ is
\begin{eqnarray*}
	\vec{X}_{2} \times_{L} \vec{Y}_{2} &=&  \vec{X} \times_{L} \vec{Y}+\varepsilon^{*}( \vec{X} \times_{L} \vec{Y}^{*} +  \vec{X}^{*} \times_{L} \vec{Y}) \\
	&=&  \vec{x}_{0} \times_{L} \vec{y}_{0} +\varepsilon( \vec{x}_{0} \times_{L} \vec{y}_{1} + \vec{x}_{1} \times_{L} \vec{y}_{0} )+\varepsilon^{*}( \vec{x}_{0} \times_{L} \vec{y}_{2} + \vec{x}_{2} \times_{L} \vec{y}_{0} ) \\
	&+& \varepsilon \varepsilon^{*} ( \vec{x}_{0} \times_{L} \vec{y}_{3} + \vec{x}_{1} \times_{L} \vec{y}_{2} +\vec{x}_{2} \times_{L} \vec{y}_{1} + \vec{x}_{3} \times_{L} \vec{y}_{0} ).
\end{eqnarray*}
The Lorentzian inner product is
\begin{eqnarray*}
	\langle \vec{X}_{2}, \vec{X}_{2} \rangle_{L}&=& \langle \vec{X}, \vec{X}\rangle_{L}+2\varepsilon^{*}\langle \vec{X}, \vec{X}^{*} \rangle_{L}  \\
	&=& \langle \vec{x}_{0}, \vec{x}_{0} \rangle_{L}+2\varepsilon\langle \vec{x}_{0}, \vec{x}_{1} \rangle_{L}+2\varepsilon^{*}\langle \vec{x}_{0}, \vec{x}_{2} \rangle_{L} \\
	&+& 2\varepsilon \varepsilon^{*}(\langle \vec{x}_{0}, \vec{x}_{3} \rangle_{L}+\langle \vec{x}_{1}, \vec{x}_{2} \rangle_{L}).
\end{eqnarray*}
The norm of $\vec{X}_{2}$ is

\begin{equation}
\begin{aligned}
\|\vec{X}_2\|_L
={}&|\vec{x}_0|_L\\
&+\varepsilon
\frac{
\operatorname{sgn}(\langle\vec{x}_0,\vec{x}_0\rangle_L)
\langle\vec{x}_0,\vec{x}_1\rangle_L
}{|\vec{x}_0|_L}\\
&+\varepsilon^*
\frac{
\operatorname{sgn}(\langle\vec{x}_0,\vec{x}_0\rangle_L)
\langle\vec{x}_0,\vec{x}_2\rangle_L
}{|\vec{x}_0|_L}\\
&+\varepsilon\varepsilon^*
\Biggl(
\frac{
\operatorname{sgn}(\langle\vec{x}_0,\vec{x}_0\rangle_L)
\left(
\langle\vec{x}_0,\vec{x}_3\rangle_L+
\langle\vec{x}_1,\vec{x}_2\rangle_L
\right)
}{|\vec{x}_0|_L}\\
&\qquad-
\frac{
\langle\vec{x}_0,\vec{x}_1\rangle_L
\langle\vec{x}_0,\vec{x}_2\rangle_L
}{|\vec{x}_0|_L^3}
\Biggr), \quad \langle \vec{x}_0, \vec{x}_0 \rangle_{L} \neq 0.
\end{aligned}
\end{equation}
If $\langle \vec{X}_{2}, \vec{X}_{2} \rangle_{L}= \pm 1$ (i.e., $\langle \vec{X}, \vec{X} \rangle_{L}= \pm1$, $\langle \vec{X}, \vec{X}^{*} \rangle_{L}= 0$), then $\vec{X}_{2}=\vec{X}+\varepsilon^{*} \vec{X}^{*}$ is called a unit hyper-dual split vector, see \cite{bekar2}.
\begin{defn} \cite{bekar2}
The set which consists of all unit hyper-dual spacelike vectors is called Lorentzian hyper-dual unit sphere:
\begin{equation*}
\hat{\mathbb{S}}_{1}^{2}=\{ \vec{X}_{2}=\vec{X}+\varepsilon^{*} \vec{X}^{*}: \langle \vec{X}_{2}, \vec{X}_{2} \rangle_{L}=1, \vec{X}_{2} \in \hat{\mathbb{D}}^{3}_{1}  \}.
\end{equation*} 
The set which consists of all unit hyper-dual timelike vectors is called hyperbolic hyper-dual unit sphere:
\begin{equation*}
	\hat{\mathbb{H}}_{0}^{2}=\{ \vec{X}_{2}=\vec{X}+\varepsilon^{*} \vec{X}^{*}: \langle \vec{X}_{2}, \vec{X}_{2} \rangle_{L}=-1, \vec{X}_{2} \in \hat{\mathbb{D}}^{3}_{1}  \}.
\end{equation*} 
\end{defn}
If \(\vec{X}_{2}=\vec{X}+\varepsilon^{*}\vec{X}^{*}\) belongs to \(\hat{\mathbb{S}}_{1}^{2}\), then \(\vec{X}\) is a unit dual spacelike vector and \(\vec{X}^{*}\) is a dual split vector. In contrast, when \(\vec{X}_{2}\in\hat{\mathbb{H}}_{0}^{2}\), the vector \(\vec{X}\) is unit dual timelike, while \(\vec{X}^{*}\) cannot be dual timelike.
\begin{thm}\label{thm}
The Lorentzian hyper-dual unit sphere $\hat{\mathbb S}^2_1$ (resp., the
hyperbolic hyper-dual unit sphere $\hat{\mathbb H}^2_0$) is isomorphic to the
tangent bundle $T\mathbb S^2_1$ of the dual Lorentzian unit sphere (resp., the
tangent bundle $T\mathbb H^2$ of the dual hyperbolic unit sphere).
\end{thm}
\begin{proof}
Let $\sigma=1$ for $\hat{\mathbb S}^2_1$ and $\sigma=-1$ for
$\hat{\mathbb H}^2_0$, and correspondingly let $\mathbb K=\mathbb S^2_1$ or
$\mathbb K=\mathbb H^2$. Consider the map
\begin{equation}\label{eq:3.2}
F:\hat{\mathbb S}^2_1\to T\mathbb S^2_1
\quad(\text{resp. } F:\hat{\mathbb H}^2_0\to T\mathbb H^2),\qquad
F(\vec X_2)=(\vec X,\vec X^{*}),
\end{equation}
where $\vec X_2=\vec X+\varepsilon^{*}\vec X^{*}$ with
$\vec X,\vec X^{*}\in\mathbb D^3_1$. \emph{$F$ is well defined.} Since $\varepsilon^{*2}=0$,
\[
\langle\vec X_2,\vec X_2\rangle_L
=\langle\vec X,\vec X\rangle_L+2\varepsilon^{*}\langle\vec X,\vec X^{*}\rangle_L .
\]
Every hyper-dual number can be written uniquely as $A+\varepsilon^{*}B$ with
$A,B\in\mathbb D$. Comparing the two components, the condition
$\langle\vec X_2,\vec X_2\rangle_L=\sigma$ is equivalent to
\[
\langle\vec X,\vec X\rangle_L=\sigma
\qquad\text{and}\qquad
\langle\vec X,\vec X^{*}\rangle_L=0 .
\]
The first equation says that $\vec X\in\mathbb K$, and the second says that
$(\vec X,\vec X^{*})\in T\mathbb K$. \emph{$F$ is bijective.} The map
$G:T\mathbb K\to\hat{\mathbb S}^2_1$ (resp.\ $\hat{\mathbb H}^2_0$),
$G(\vec X,\vec Y)=\vec X+\varepsilon^{*}\vec Y$, is well defined by the same
equivalence, and clearly $G\circ F=\mathrm{id}$ and $F\circ G=\mathrm{id}$. \emph{$F$ is an isomorphism of bundles.} Writing
$\vec X=\vec x_0+\varepsilon\vec x_1$ and
$\vec X^{*}=\vec x_2+\varepsilon\vec x_3$, both $F$ and $G$ act as the identity
on the real coordinates $(\vec x_0,\vec x_1,\vec x_2,\vec x_3)$; hence they are
smooth and mutually inverse. Moreover, $F$ maps the set of points of the
hyper-dual sphere with a fixed first component $\vec X$ onto the fibre
\[
T_{\vec X}\mathbb K=\{\vec Y\in\mathbb D^3_1:\ \langle\vec X,\vec Y\rangle_L=0\},
\]
and is $\mathbb D$-linear in the second component. Hence $F$ is an isomorphism
of $\hat{\mathbb S}^2_1$ onto $T\mathbb S^2_1$ (resp.\ of
$\hat{\mathbb H}^2_0$ onto $T\mathbb H^2$), where $\vec X$ and $\vec X^{*}$
are vectors in $\mathbb D^3_1$.
\end{proof}
\begin{thm}[E. Study Hyper-Map in Minkowski Space \cite{gok} ]
Each point of the Lorentzian unit hyper-dual sphere \(\hat{\mathbb{S}}_{1}^{2}\) (resp., the hyperbolic unit hyper-dual sphere \(\hat{\mathbb{H}}_{0}^{2}\)) determines an oriented dual spacelike line (resp., an oriented dual timelike line) in \(\mathbb{D}_{1}^{3}\).
\end{thm}
Consequently, we write
\begin{align*}
\hat{\mathbb{S}}^2_1 &\;\rightarrow\; T\mathbb{S}^2_1 \;\rightarrow\; T(TS^2_1),\\
\hat{\mathbb{H}}^2_0 &\;\rightarrow\; T\mathbb{H}^2 \;\rightarrow\; T(TH^2_0),
\end{align*}
where $T(TS^2_1)$ and $T(TH^2_0)$ denote the tangent bundles of $TS^2_1$ and
$TH^2_0$, respectively. \\
Let $S_1^2$ and $H_0^2$ denote the
Lorentzian and hyperbolic unit spheres
in $\mathbb{R}_1^3$, respectively.
Their dual counterparts are denoted
by $\mathbb{S}_1^2$ and $\mathbb{H}^2$. We define the unit spacelike
tangent bundle of $S_1^2$ by
\[
T_1^+S_1^2
=
\left\{
(\vec{x},\vec{u})\in TS_1^2:
\langle\vec{u},\vec{u}\rangle_L=1
\right\}.
\]
Similarly, the unit tangent
bundle of $H_0^2$ is defined by
\[
T_1^+H_0^2
=
\left\{
(\vec{x},\vec{u})\in TH_0^2:
\langle\vec{u},\vec{u}\rangle_L=1
\right\}.
\]
For the dual Lorentzian unit sphere,
we introduce
\[
T_1^+\mathbb{S}_1^2
=
\left\{
(\vec{X},\vec{Y})
\in T\mathbb{S}_1^2:
\langle\vec{Y},\vec{Y}\rangle_L=1
\right\}.
\]
For the dual hyperbolic unit sphere,
we define
\[
T_1^+\mathbb{H}^2
=
\left\{
(\vec{X},\vec{Y})
\in T\mathbb{H}^2:
\langle\vec{Y},\vec{Y}\rangle_L=1
\right\}.
\]
\begin{cor}
\label{cor:unit_lorentzian_hyperdual}
Let $\hat{\mathbb{S}}_1^2$ be the
Lorentzian hyper-dual unit sphere.
Define
\[
U\hat{\mathbb{S}}_1^2
=
\left\{
\vec{X}_2
=
\vec{X}+\varepsilon^*\vec{X}^*
\in\hat{\mathbb{S}}_1^2:
\langle\vec{X}^*,\vec{X}^*\rangle_L=1
\right\}.
\]
Then there exist natural
diffeomorphisms
\[
U\hat{\mathbb{S}}_1^2
\cong
T_1^+\mathbb{S}_1^2
\cong
T(T_1^+S_1^2).
\]
\end{cor}
\begin{cor}
\label{cor:unit_hyperbolic_hyperdual}
Let $\hat{\mathbb{H}}_0^2$ be the
hyperbolic hyper-dual unit sphere.
Define
\[
U\hat{\mathbb{H}}_0^2
=
\left\{
\vec{X}_2
=
\vec{X}+\varepsilon^*\vec{X}^*
\in\hat{\mathbb{H}}_0^2:
\langle\vec{X}^*,\vec{X}^*\rangle_L=1
\right\}.
\]
Then there exist natural
diffeomorphisms
\[
U\hat{\mathbb{H}}_0^2
\cong
T_1^+\mathbb{H}^2
\cong
T(T_1^+H_0^2).
\]
\end{cor}
\begin{rem}
The notation $U\hat{\mathbb{S}}_1^2$ called the subset
corresponding to unit spacelike
tangent vectors. Since the Lorentzian metric on
$S_1^2$ is indefinite, the
unit timelike tangent vectors
form a distinct subset. On the other hand, the induced
metric on $H_0^2$ is positive
definite. Consequently, every
nonzero tangent vector to $H_0^2$
is spacelike, and the corresponding
unit tangent condition has
positive Lorentzian squared norm.
\end{rem}

\section{Non-null ruled surfaces in terms of pseudo hyper-dual tangent sphere bundles }
In this section, non-null ruled surfaces in $\mathbb{D}_{1}^{3}$ are introduced and classified. Moreover, their developability conditions are discussed.
\begin{defn} \label{41}
A dual ruled surface corresponding the hyper-dual Lorentzian curve $\Gamma(t)=(\vec{C}(t), \vec{C}^{*}(t))$ in $\mathbb{D}_{1}^{3}$ is defined as
\begin{equation*}
\phi(t,V)= \beta(t)+V\vec{C}(t)
\end{equation*}
where $\beta(t)$ and $\vec{C}(t)$ are the base curve and ruling, respectively. Here, $V=v+\varepsilon v^{*}$ is a dual number in $\mathbb{D}$.
\end{defn}
\begin{thm}\label{42}
Let $\Gamma(t)=\vec C(t)+\varepsilon^{*}\vec C^{*}(t)$ be a smooth hyper-dual
Lorentzian curve on $\hat{\mathbb S}^2_1$ (resp.\ $\hat{\mathbb H}^2_0$). Then
$\Gamma$ determines a unique dual ruled surface in $\mathbb D^3_1$,
\[
\phi(t,V)=\vec C(t)\times_L\vec C^{*}(t)+V\,\vec C(t),
\]
whose base curve is $\beta=\vec C\times_L\vec C^{*}$ and whose ruling is the unit
dual vector $\vec C$. This gives the maps
\[
\hat{\mathbb S}^2_1\to RS(\mathbb D^3_1),\qquad
\hat{\mathbb H}^2_0\to RS(\mathbb D^3_1),\qquad \Gamma\mapsto\phi ,
\]
where $RS(\mathbb D^3_1)$ denotes the set of dual ruled surfaces in $\mathbb D^3_1.$
\end{thm}
\subsection{Causal classification of ruled surfaces
associated with hyper-dual Lorentzian curves}

We now establish a causal classification of the
dual ruled surfaces generated by hyper-dual
Lorentzian curves. \\
Let $\Gamma(t)=\vec{C}(t)+\varepsilon^{*}\vec{C}^{*}(t)$ be a smooth hyper-dual Lorentzian curve on
$\hat{\mathbb{S}}_{1}^{2}$ or $\hat{\mathbb{H}}_{0}^{2}$, where
\[
\vec{C}(t)=\vec{c}_0(t)+\varepsilon\vec{c}_1(t),
\qquad
\vec{C}^{*}(t)=\vec{c}_2(t)+\varepsilon\vec{c}_3(t).
\]
Let
\[
\sigma=
\begin{cases}
1, & \Gamma(t)\in \hat{\mathbb{S}}_{1}^{2},\\
-1, & \Gamma(t)\in \hat{\mathbb{H}}_{0}^{2}.
\end{cases}
\]
The unit hyper-dual condition
\[
\langle\Gamma(t),\Gamma(t)\rangle_L=\sigma
\]
implies
\[
\langle\vec{C},\vec{C}\rangle_L=\sigma,
\qquad
\langle\vec{C},\vec{C}^{*}\rangle_L=0.
\]
Consequently, we obtain
\[
\begin{aligned}
\langle\vec{c}_0,\vec{c}_0\rangle_L&=\sigma,\\
\langle\vec{c}_0,\vec{c}_1\rangle_L&=0,\\
\langle\vec{c}_0,\vec{c}_2\rangle_L&=0,\\
\langle\vec{c}_0,\vec{c}_3\rangle_L
+\langle\vec{c}_1,\vec{c}_2\rangle_L&=0.
\end{aligned}
\]
By Definition \ref{41}, the associated dual
ruled surface is
\[
\phi(t,V)
=\beta(t)+V\vec{C}(t),
\]
where
$\beta(t)=\vec{C}(t)\times_L\vec{C}^{*}(t)$ and $V=v+\varepsilon v^{*}.$
Hence, we get
\[
\begin{aligned}
\beta_0(t)
&=\vec{c}_0(t)\times_L\vec{c}_2(t),\\
\beta_1(t)
&=\vec{c}_1(t)\times_L\vec{c}_2(t)
+\vec{c}_0(t)\times_L\vec{c}_3(t).
\end{aligned}
\]
Then the non-null ruled surface can be written as follows:
\[
\varphi(t,V)=\varphi_0(t,v)
+\varepsilon\varphi_1(t,v,v^{*}),
\]
where
\[
\begin{aligned}
\varphi_0(t,v)
&=\beta_0(t)+v\vec{c}_0(t),\\
\varphi_1(t,v,v^{*})
&=\beta_1(t)+v\vec{c}_1(t)
+v^{*}\vec{c}_0(t).
\end{aligned}
\]
We consider the dual first fundamental form
\[
\widetilde{I}
=\widetilde{E}\,dt^2
+2\widetilde{F}\,dt\,dV
+\widetilde{G}\,dV^2,
\]
whose coefficients are
\[
\begin{aligned}
\widetilde{E}
&=\langle\varphi_t,\varphi_t\rangle_L,\\
\widetilde{F}
&=\langle\varphi_t,\varphi_V\rangle_L,\\
\widetilde{G}
&=\langle\varphi_V,\varphi_V\rangle_L.
\end{aligned}
\]
The causal type of this dual metric is
defined by the signature of its reduction
modulo $\varepsilon$.
Thus, it is called spacelike when its
real-part metric is positive definite,
and timelike when its real-part metric
has Lorentzian signature.

\begin{prop}
\label{prop:dual_causal_classification}
Let $\Gamma(t)=\vec{C}(t)
+\varepsilon^{*}\vec{C}^{*}(t)$
be a smooth hyper-dual Lorentzian curve
on $\hat{\mathbb{S}}_{1}^{2}$ or $\hat{\mathbb{H}}_{0}^{2}$,
and let $\varphi(t,V)
=\vec{C}(t)\times_L\vec{C}^{*}(t)
+V\vec{C}(t)$
be its associated dual ruled surface. Let
\[
\begin{aligned}
A_0(t,v)
&=\beta_0'(t)+v\vec{c}_0'(t),\\
A_1(t,v,v^{*})
&=\beta_1'(t)+v\vec{c}_1'(t)
+v^{*}\vec{c}_0'(t).
\end{aligned}
\]
Define
\[
\begin{aligned}
\Delta_0(t,v)
={}&\sigma\langle A_0,A_0\rangle_L
-\langle A_0,\vec{c}_0\rangle_L^2
\end{aligned}
\]
and
\[
\begin{aligned}
\Delta_1(t,v,v^{*})
={}&2\sigma\langle A_0,A_1\rangle_L-2\langle A_0,\vec{c}_0\rangle_L
\big(
\langle A_1,\vec{c}_0\rangle_L
+\langle A_0,\vec{c}_1\rangle_L
\big).
\end{aligned}
\]
Then the determinant of the dual first
fundamental form satisfies
\[
\widetilde{E}\widetilde{G}
-\widetilde{F}^{\,2}
=\Delta_0+\varepsilon\Delta_1.
\]
At every point where the associated
real-part surface $\varphi_0$ is a
regular immersion, the following
statements hold:

\begin{enumerate}

\item If
$\Gamma(t)\in\hat{\mathbb{H}}_0^2$,
then
\[
\Delta_0(t,v)<0.
\]
Consequently, the induced dual metric
has timelike causal type.
\item If
$\Gamma(t)\in \hat{\mathbb{S}}_1^2$,
then the induced dual metric has
spacelike causal type if and only if
\[
\Delta_0(t,v)>0.
\]

\item If
$\Gamma(t)\in \hat{\mathbb{S}}_{1}^{2}$,
then the induced dual metric has
timelike causal type if and only if
\[
\Delta_0(t,v)<0.
\]

\item The induced dual metric is
non-degenerate over $D$ if and only if
\[
\Delta_0(t,v)\neq 0.
\]
In particular, if
\[
\Delta_0(t,v)=0,
\]
then the real-part metric is degenerate
and the determinant of the dual metric
is not invertible in $D$.

\end{enumerate}
\end{prop}

\begin{proof}
Differentiating the parametrization
with respect to $t$ and the dual
parameter $V$, we obtain
\[
\begin{aligned}
\varphi_t
&=\beta'(t)+V\vec{C}'(t)=A_0+\varepsilon A_1,
\end{aligned}
\]
and
\[
\varphi_V
=\vec{C}(t)
=\vec{c}_0+\varepsilon\vec{c}_1.
\]
Therefore, we write
\[
\begin{aligned}
\widetilde{E}
={}&\langle A_0,A_0\rangle_L
+2\varepsilon
\langle A_0,A_1\rangle_L,
\end{aligned}
\]
\[
\begin{aligned}
\widetilde{F}
={}&\langle A_0,\vec{c}_0\rangle_L+\varepsilon\big(
\langle A_1,\vec{c}_0\rangle_L
+\langle A_0,\vec{c}_1\rangle_L
\big),
\end{aligned}
\]
and
\[
\begin{aligned}
\widetilde{G}
&=\langle\vec{C},\vec{C}\rangle_L=\sigma.
\end{aligned}
\]
Substituting these expressions into
the determinant of the dual metric
gives
\[
\widetilde{E}\widetilde{G}
-\widetilde{F}^{\,2}
=\Delta_0+\varepsilon\Delta_1,
\]
where $\Delta_0$ and $\Delta_1$
are defined as above. \\
Suppose first that
$\Gamma(t)\in \hat{\mathbb{H}}_{0}^{2}$.
Then
\[
\langle\vec{c}_0,\vec{c}_0\rangle_L=-1.
\]
At every regular point, the real
tangent plane is spanned by the
linearly independent vectors
$A_0$ and $\vec{c}_0$.
Since this plane contains a timelike
vector, its induced metric has
Lorentzian signature.
Consequently,
\[
\Delta_0(t,v)<0.
\]
Next, suppose that
$\Gamma(t)\in \hat{\mathbb{S}}_{1}^{2}$.
Then
\[
\langle\vec{c}_0,\vec{c}_0\rangle_L=1.
\]
The matrix of the real-part induced
metric is
\[
g_0=
\begin{pmatrix}
\langle A_0,A_0\rangle_L
&
\langle A_0,\vec{c}_0\rangle_L\\
\langle A_0,\vec{c}_0\rangle_L
&
1
\end{pmatrix}.
\]
Its determinant is
\[
\det(g_0)=\Delta_0(t,v).
\]
Since the lower-right entry is
positive, $g_0$ is positive definite
if and only if
\[
\Delta_0(t,v)>0.
\]
Similarly, $g_0$ has Lorentzian
signature if and only if
\[
\Delta_0(t,v)<0.
\]
Finally, recall that a dual number
$a+\varepsilon b$ is invertible in
$D$ if and only if $a\neq 0$.
Thus,
\[
\widetilde{E}\widetilde{G}
-\widetilde{F}^{\,2}
\]
is invertible precisely when
\[
\Delta_0(t,v)\neq 0.
\]
If $\Delta_0(t,v)=0$, the determinant
reduces to $\varepsilon\Delta_1,$
which is not invertible in $D$. \\
This completes the proof.
\end{proof}
\begin{cor}\label{cor:nonnull}
Under the hypotheses of Theorem~\ref{42}, put
$U_\Gamma=\{(t,v):\Delta_0(t,v)\neq 0\}$. Then:
\begin{enumerate}
\item[(i)] If $\Gamma(t)\in\hat{\mathbb H}^2_0$, then $U_\Gamma$ is exactly the set of
regular points of $\phi_0$, and $\phi$ is timelike on $U_\Gamma$.
\item[(ii)] If $\Gamma(t)\in\hat{\mathbb S}^2_1$, then $\phi$ is non-null on $U_\Gamma$:
it is spacelike where $\Delta_0>0$ and timelike where $\Delta_0<0$.
\end{enumerate}
In particular, $\phi|_{U_\Gamma}$ is a non-null dual ruled surface, so
$\Gamma\mapsto\phi|_{U_\Gamma}$ takes values in the subset $RS^{*}(\mathbb D^3_1) \subset RS(\mathbb D^3_1)$.
\end{cor}
\begin{proof}
Set $A_0^{\perp}=A_0-\sigma\langle A_0,\vec c_0\rangle_L\vec c_0$. Then
$\langle A_0^{\perp},\vec c_0\rangle_L=0$ and
$\Delta_0=\sigma\langle A_0^{\perp},A_0^{\perp}\rangle_L$. A point is regular for
$\phi_0$ iff $A_0^{\perp}\neq0$, so $\Delta_0\neq0$ implies regularity.
If $\sigma=-1$, the vector $A_0^{\perp}$ is orthogonal to the timelike vector
$\vec c_0$, hence spacelike or zero, so regularity is equivalent to
$\Delta_0<0$. The remaining claims follow from Proposition \ref{prop:dual_causal_classification}.
\end{proof}
\begin{rem}\label{rem:degenerate}
The restriction to $U_\Gamma$ in Corollary \ref{cor:nonnull} cannot be dropped in the case of
$\hat{\mathbb S}^2_1$. Let $\vec e_1,\vec e_2,\vec e_3$ be the standard basis of
$\mathbb R^3_1$ and consider
\[
\vec c_0=\vec e_1,\qquad \vec c_1=\vec 0,\qquad
\vec c_2(t)=t(\vec e_2+\vec e_3),\qquad \vec c_3=\vec 0,
\]
that is,
$\Gamma(t)=\vec e_1+\varepsilon^{*}\,t(\vec e_2+\vec e_3)$.
Since $\langle\vec c_0,\vec c_0\rangle_L=1$,
$\langle\vec c_0,\vec c_1\rangle_L=\langle\vec c_0,\vec c_2\rangle_L=0$ and
$\langle\vec c_0,\vec c_3\rangle_L+\langle\vec c_1,\vec c_2\rangle_L=0$, we have
$\Gamma(t)\in\hat{\mathbb S}^2_1$.
Using $\vec e_1\times_L\vec e_2=-\vec e_3$ and
$\vec e_1\times_L\vec e_3=-\vec e_2$, we obtain
\[
\beta_0(t)=\vec c_0\times_L\vec c_2=-t(\vec e_2+\vec e_3),\qquad
\beta_1(t)=\vec 0 ,
\]
so that
\[
\phi_0(t,v)=-t(\vec e_2+\vec e_3)+v\,\vec e_1,\qquad
\phi_1(t,v,v^{*})=v^{*}\vec e_1 .
\]
Hence $A_0=-(\vec e_2+\vec e_3)$ and $A_1=\vec 0$. The vectors $A_0$ and
$\vec c_0$ are linearly independent, so $\phi_0$ is a regular immersion.
However,
\[
\langle A_0,A_0\rangle_L=0,\qquad \langle A_0,\vec c_0\rangle_L=0,
\]
which gives $\Delta_0\equiv 0$ and $\Delta_1\equiv 0$. Thus
$\tilde E\tilde G-\tilde F^2\equiv 0$, and $\phi_0$ parametrizes the lightlike
plane spanned by the spacelike vector $\vec e_1$ and the null vector
$\vec e_2+\vec e_3$. In this case $U_\Gamma=\emptyset$, and the dual ruled
surface associated with $\Gamma$ is degenerate at every point, although it is regular. This shows that, for curves on $\hat{\mathbb S}^2_1$, regularity alone
does not guarantee that the associated dual ruled surface is non-null.
By contrast, no such example exists on $\hat{\mathbb H}^2_0$: by
Corollary \ref{cor:nonnull} (i), every regular point of $\phi_0$ lies in $U_\Gamma$.
\end{rem}
Now, we will give the geometric interpretations for the developability condition of the dual ruled surfaces:
\begin{prop} \label{38}
The corresponding dual ruled surface of hyper-dual Lorentzian curve $\Gamma(t)=(\vec{C}(t), \vec{C}^{*}(t))$ in Lorentzian hyper-dual sphere (resp., hyperbolic hyper-dual unit sphere) is developable iff
\begin{equation*}
\langle c_{0}^{'}, c_{2}^{'} \rangle_{L}=0 \ \ \text{and} \ \ \langle c_{0}^{'}, c_{3}^{'} \rangle_{L}=-\langle c_{1}^{'}, c_{2}^{'} \rangle_{L}.
\end{equation*}
\end{prop}
\begin{proof}
Assume that $\phi(t,V)= \vec{C}(t) \times_{L} \vec{C}^{*}(t)+V\vec{C}(t)$ is a corresponding dual ruled surface of hyper-dual Lorentzian curve $\Gamma(t)=(\vec{C}(t), \vec{C}^{*}(t))$ in Lorentzian hyper-dual sphere (resp., hyperbolic hyper-dual unit sphere). Therefore, we have
\begin{equation*}
\det((\vec{C} \times_{L} \vec{C}^{*})^{'}, \vec{C}^{'}, \vec{C})=\sigma \langle \vec{C}^{'}, \vec{C}^{* '} \rangle_{L},
\end{equation*}
where $\sigma=\langle \vec{C}, \vec{C} \rangle_{L} \in \{ -1,1 \} $
Exploiting the developability condition, the last equation gives
\begin{equation*}
\langle c_{0}^{'}, c_{2}^{'} \rangle_{L}+ \varepsilon(\langle c_{0}^{'}, c_{3}^{'} \rangle_{L} + \langle c_{1}^{'}, c_{2}^{'} \rangle_{L})=0 
\end{equation*}
and
\begin{equation*}
\langle c_{0}^{'}, c_{2}^{'} \rangle_{L}=0 \ \ \text{and} \ \  \langle c_{0}^{'}, c_{3}^{'} \rangle_{L} + \langle c_{1}^{'}, c_{2}^{'} \rangle_{L}=0. 
\end{equation*}
This completes the proof.
\end{proof}
\begin{cor}
Let
\[
\Gamma(t)=\vec{C}(t)
+\varepsilon^{*}\vec{C}^{*}(t)
\]
be a smooth hyper-dual Lorentzian curve
on $\hat{\mathbb{S}}_1^2$
(resp., $\hat{\mathbb{H}}_0^2$). Assume that the real part of
$\vec{C}'(t)$ is non-null. Then the associated dual ruled
surface
\[
\varphi(t,V)
=\vec{C}(t)\times_L\vec{C}^{*}(t)
+V\vec{C}(t)
\]
is developable iff
$\Gamma(t)$ and $\vec{C}(t)$
have the same arc-length function,
where the dual-valued arc length
of $\vec{C}$ is naturally embedded
in the hyper-dual numbers.
\end{cor}

\begin{proof}
Since
\[
\Gamma'
=\vec{C}'+\varepsilon^{*}\vec{C}^{*'},
\]
we obtain
\[
\|\Gamma'\|_L
=
\|\vec{C}'\|_L
+
\varepsilon^{*}
\frac{
\eta\langle
\vec{C}',\vec{C}^{*'}
\rangle_L
}{
\|\vec{C}'\|_L
},
\]
where
\[
\eta=
\operatorname{sgn}
\langle\vec{c}_0',\vec{c}_0'\rangle_L.
\]
Since $\|\vec{C}'\|_L$ is
invertible in $\mathbb{D}$,
it follows that
\[
\|\Gamma'\|_L=\|\vec{C}'\|_L
\]
if and only if
\[
\langle
\vec{C}',\vec{C}^{*'}
\rangle_L=0.
\]
By Proposition \ref{38}, this is
equivalent to the developability
of the associated dual ruled
surface. Integrating the equality of
the speeds with the same initial
value gives the equality of
the corresponding arc-length
functions. \\
This completes the proof.
\end{proof}

\begin{cor}
 Let \(\Gamma(t)=(\vec C(t),\vec C^{*}(t))\) be a regular hyper-dual Lorentzian curve on $\hat{\mathbb{S}}_{1}^{2}$ (resp., $\hat{\mathbb{H}}_{0}^{2}$). Then the corresponding dual ruled surface
\[
\phi(t,V)=\vec C(t)\times_L\vec C^{*}(t)+V\vec C(t)
\]
is developable if and only if
\[
\langle\vec C'(t),\vec C^{*'}(t)\rangle_L=0.
\]
Equivalently, the tangent vectors of the two dual Lorentzian curves are orthogonal with respect to the dual Lorentzian inner product.
\end{cor}
\begin{cor}
 Let \(\Gamma(t)=(\vec C(t),\vec C^{*}(t))\) be a smooth hyper-dual Lorentzian curve on $\hat{\mathbb{S}}_{1}^{2}$ (resp., $\hat{\mathbb{H}}_{0}^{2}$). Assume that \(\langle\vec C'(t),\vec C'(t)\rangle_L\) is invertible. If
\[
\left\langle
\big(\vec C(t)\times_L\vec C^{*}(t)\big)',
\vec C'(t)
\right\rangle_L=0,
\]
then the striction curve and the base curve of the corresponding dual ruled surface
\[
\phi(t,V)=\vec C(t)\times_L\vec C^{*}(t)+V\vec C(t)
\]
coincide.
\end{cor}

\begin{prop}
\label{prop:associated_ruled_surfaces}

Let
\[
\Gamma(t)
=\vec{C}(t)+\varepsilon^{*}\vec{C}^{*}(t)
\]
be a smooth hyper-dual Lorentzian curve
on $\hat{\mathbb{S}}_{1}^{2}$ (resp., $\hat{\mathbb{H}}_{0}^{2}$),
where
$\vec{C}(t)
=\vec{c}_0(t)+\varepsilon\vec{c}_1(t)
$
and
$\vec{C}^{*}(t)
=\vec{c}_2(t)+\varepsilon\vec{c}_3(t).$
Consider the associated dual ruled surface
\[
\varphi(t,V)
=\vec{C}(t)\times_L\vec{C}^{*}(t)
+V\vec{C}(t).
\]
Then its real and dual components
are given by
\[
\varphi(t,V)
=\varphi_0(t,v)
+\varepsilon\varphi_1(t,v,v^{*}),
\]
where
\[
\varphi_0(t,v)
=\beta_0(t)+v\vec{c}_0(t)
\]
and
\[
\varphi_1(t,v,v^{*})
=\beta_1(t)+v\vec{c}_1(t)
+v^{*}\vec{c}_0(t),
\]
with
\[
\beta_0(t)
=\vec{c}_0(t)\times_L\vec{c}_2(t),
\]
\[
\beta_1(t)
=\vec{c}_1(t)\times_L\vec{c}_2(t)
+\vec{c}_0(t)\times_L\vec{c}_3(t).
\]
Moreover, the ruling directions
$\vec{c}_0(t)$ and $\vec{c}_1(t)$
satisfy
\[
\langle
\vec{c}_0(t),\vec{c}_1(t)
\rangle_L=0.
\]
In particular, for each fixed
$k\in\mathbb{R}$, the parametrization
\[
\psi_k(t,v)
=\beta_1(t)+k\vec{c}_0(t)
+v\vec{c}_1(t)
\]
is a ruled parametrization whenever
$\vec{c}_1(t)\neq0$.
If $\langle
\vec{c}_1(t),\vec{c}_1(t)
\rangle_L\neq0,$
then $\psi_k$ has a non-null
ruling direction orthogonal to
that of $\varphi_0$. \\
Whenever both parametrizations
are regular immersions and their
induced metrics are non-degenerate,
they define two non-degenerate
ruled surfaces in $\mathbb{R}_1^3$
with mutually orthogonal
ruling directions.
\end{prop}

\begin{proof}
By the definition of the associated
dual ruled surface, we have
\[
\varphi(t,V)
=
\vec{C}(t)\times_L\vec{C}^{*}(t)
+V\vec{C}(t).
\]
Hence, we write
\[
\begin{aligned}
\varphi(t,V)
={}&
(\vec{c}_0+\varepsilon\vec{c}_1)
\times_L
(\vec{c}_2+\varepsilon\vec{c}_3)\\
+&
(v+\varepsilon v^{*})
(\vec{c}_0+\varepsilon\vec{c}_1).
\end{aligned}
\]
Since $\varepsilon^2=0$,
the above expression becomes
\[
\begin{aligned}
\varphi(t,V)
={}&
\vec{c}_0\times_L\vec{c}_2
+v\vec{c}_0\\
&+\varepsilon\Big(
\vec{c}_1\times_L\vec{c}_2
+\vec{c}_0\times_L\vec{c}_3
+v\vec{c}_1+v^{*}\vec{c}_0
\Big).
\end{aligned}
\]
Consequently, we have
\[
\varphi(t,V)
=\varphi_0(t,v)
+\varepsilon\varphi_1(t,v,v^{*}),
\]
where
\[
\varphi_0(t,v)
=\beta_0(t)+v\vec{c}_0(t)
\]
and
\[
\varphi_1(t,v,v^{*})
=\beta_1(t)+v\vec{c}_1(t)
+v^{*}\vec{c}_0(t).
\]
On the other hand, since
$\Gamma(t)$ belongs to
$\hat{\mathbb{S}}_1^2$ or
$\hat{\mathbb{H}}_0^2$, we have
\[
\langle
\Gamma(t),\Gamma(t)
\rangle_L=\sigma,
\qquad
\sigma\in\{1,-1\}.
\]
Therefore, we write
\[
\langle
\vec{C}(t),\vec{C}(t)
\rangle_L=\sigma.
\]
Expanding the inner product
yields
\[
\begin{aligned}
\langle
\vec{C},\vec{C}
\rangle_L
=
\langle\vec{c}_0,\vec{c}_0\rangle_L
+2\varepsilon
\langle\vec{c}_0,\vec{c}_1\rangle_L.
\end{aligned}
\]
Comparing the real and dual
coefficients gives $\langle
\vec{c}_0,\vec{c}_0
\rangle_L=\sigma$
and
$\langle
\vec{c}_0,\vec{c}_1
\rangle_L=0.$
Thus, $\vec{c}_0$ is a non-null
unit vector, and $\vec{c}_1$
is Lorentzian orthogonal to
$\vec{c}_0$.
For any fixed $v^{*}=k$,
the dual component reduces to
\[
\psi_k(t,v)
=
\beta_1(t)+k\vec{c}_0(t)
+v\vec{c}_1(t).
\]
If $\vec{c}_1(t)\neq0$,
this expression has the form
of a ruled parametrization
with ruling direction
$\vec{c}_1(t)$. 
If $\vec{c}_1$ is non-null,
the two ruling directions
$\vec{c}_0$ and $\vec{c}_1$
are both non-null and mutually
orthogonal. \\
Finally, under the stated
regularity and non-degeneracy
assumptions, these parametrizations
define non-degenerate ruled
surfaces in $\mathbb{R}_1^3$. \\
This completes the proof.
\end{proof}

\begin{rem}
The preceding proposition does not imply
that the two associated ruled surfaces
necessarily possess a common base curve.
Indeed, their base curves are given by
\[
\beta_0(t)
=\vec{c}_0(t)\times_L\vec{c}_2(t)
\]
and
\[
\beta_1(t)+k\vec{c}_0(t)
=
\vec{c}_1(t)\times_L\vec{c}_2(t)
+\vec{c}_0(t)\times_L\vec{c}_3(t)
+k\vec{c}_0(t),
\]
which need not coincide. Suppose, in addition, that
$\vec{c}_1(t)$ is non-null and
\[
\beta_1(t)-\beta_0(t)
\in
\operatorname{span}
\{\vec{c}_0(t),\vec{c}_1(t)\}
\]
for every $t$ in the parameter interval. Since
\[
\langle\vec{c}_0,\vec{c}_1\rangle_L=0
\]
and both vectors are non-null,
they are linearly independent. \\
Consequently, there exist unique
smooth functions $f(t)$ and $g(t)$
such that
\[
\beta_1(t)-\beta_0(t)
=
f(t)\vec{c}_0(t)-g(t)\vec{c}_1(t).
\]
These functions are explicitly given by
\[
f(t)
=
\frac{
\langle\beta_1(t)-\beta_0(t),
\vec{c}_0(t)\rangle_L
}{
\langle\vec{c}_0(t),
\vec{c}_0(t)\rangle_L
}
\]
and
\[
g(t)
=
-\frac{
\langle\beta_1(t)-\beta_0(t),
\vec{c}_1(t)\rangle_L
}{
\langle\vec{c}_1(t),
\vec{c}_1(t)\rangle_L
}.
\]
For each fixed $k\in\mathbb{R}$,
define
\[
\begin{aligned}
M_k(t)
&=\beta_0(t)
+\big(f(t)+k\big)\vec{c}_0(t)\\
&=\beta_1(t)
+k\vec{c}_0(t)
+g(t)\vec{c}_1(t).
\end{aligned}
\]
Thus, the two non-null ruled surfaces admit
the representations
\[
\overline{\varphi}_{0,k}(t,v)
=
M_k(t)+v\vec{c}_0(t)
\]
and
\[
\overline{\psi}_k(t,w)
=
M_k(t)+w\vec{c}_1(t).
\]
These parametrizations describe
the original ruled surfaces after
suitable shifts of their ruling
parameters. 
Hence, under the additional
assumption above, the two non-null ruled
surfaces possess a common base
curve $M_k(t)$ and mutually
orthogonal ruling directions. 
If both parametrizations are
regular immersions with
non-degenerate induced metrics,
they define a pair of non-degenerate non-null
ruled surfaces in $\mathbb{R}_1^3$. 
Therefore, the existence of a
common base curve follows from
the additional span condition,
whereas the orthogonality of the
ruling directions follows directly
from the hyper-dual unit condition.
\end{rem}

\begin{prop}
\label{412}

Let $\phi(t,v)=M(t)+v\vec{c}(t)$ and $\phi_1(t,w)=M(t)+w\vec{c}^{*}(t)$
be two smooth ruled surfaces in
$\mathbb{R}_1^3$ with a common
base curve $M(t)$. Assume that their ruling directions
$\vec{c}(t)$ and $\vec{c}^{*}(t)$
are everywhere non-null. Then, after normalizing the
ruling directions, the two
surfaces determine the hyper-dual curves encoding the rulings of the two surfaces, $\Gamma(t)
=\vec{C}(t)+\varepsilon^{*}
\vec{C}^{*}(t)$ and $\Gamma_1(t)
=\vec{C}^{*}(t)+\varepsilon^{*}
\vec{C}(t)$ in $\hat{\mathbb{D}}_2$, where $\vec{C}(t)
=\vec{c}(t)
+\varepsilon
\big(M(t)\times_L\vec{c}(t)\big)$
and
$\vec{C}^{*}(t)
=\vec{c}^{*}(t)
+\varepsilon
\big(M(t)\times_L\vec{c}^{*}(t)\big).$
Here, $\vec{c}$ and $\vec{c}^{*}$
denote the normalized ruling
directions. Furthermore, if
\[
\langle\vec{c}(t),
\vec{c}^{*}(t)\rangle_L=0,
\]
then $\Gamma$ and $\Gamma_1$
belong to the corresponding
unit hyper-dual Lorentzian
or hyperbolic spheres,
according to the causal
characters of their respective
ruling directions.
\end{prop}

\begin{proof}
Since the ruling directions
are non-null,
they can be normalized locally
by
\[
\vec{c}\longmapsto
\frac{\vec{c}}
{\sqrt{|\langle\vec{c},
\vec{c}\rangle_L|}}
\]
and
\[
\vec{c}^{*}\longmapsto
\frac{\vec{c}^{*}}
{\sqrt{|\langle\vec{c}^{*},
\vec{c}^{*}\rangle_L|}}.
\]
Such a normalization changes
only the ruling parameters,
not the corresponding
geometric surfaces. Hence, without loss of
generality, we may assume
\[
\langle\vec{c},\vec{c}\rangle_L
=\sigma,
\qquad
\langle\vec{c}^{*},
\vec{c}^{*}\rangle_L
=\tau,
\]
where
\[
\sigma,\tau\in\{1,-1\}.
\]
By the E. Study mapping,
the oriented ruling lines
of the first surface determine
the dual vector
\[
\vec{C}
=\vec{c}
+\varepsilon(M\times_L\vec{c}),
\]
while those of the second
surface determine
\[
\vec{C}^{*}
=\vec{c}^{*}
+\varepsilon(M\times_L\vec{c}^{*}).
\]
Using the Lorentzian
orthogonality of a vector
and its vector product,
we have
\[
\langle
\vec{c},M\times_L\vec{c}
\rangle_L=0
\]
and
\[
\langle
\vec{c}^{*},
M\times_L\vec{c}^{*}
\rangle_L=0.
\]
Consequently, we have
\[
\begin{aligned}
\langle\vec{C},\vec{C}\rangle_L
&=
\langle\vec{c},\vec{c}\rangle_L+
2\varepsilon
\langle\vec{c},
M\times_L\vec{c}\rangle_L=\sigma.
\end{aligned}
\]
Similarly, we write
\[
\langle\vec{C}^{*},
\vec{C}^{*}\rangle_L=\tau.
\]
Thus, $\vec{C}$ and
$\vec{C}^{*}$ are unit
dual Lorentzian vectors. Moreover, the expressions $\Gamma
=\vec{C}+\varepsilon^{*}\vec{C}^{*}$
and $\Gamma_1
=\vec{C}^{*}+\varepsilon^{*}\vec{C}$
define hyper-dual curves
in $\hat{\mathbb{D}}_2$. Now suppose additionally that $\langle\vec{c},
\vec{c}^{*}\rangle_L=0.$
Expanding the dual inner product
gives
\[
\begin{aligned}
\langle\vec{C},
\vec{C}^{*}\rangle_L
={}&
\langle\vec{c},
\vec{c}^{*}\rangle_L+\varepsilon\Big(
\langle\vec{c},
M\times_L\vec{c}^{*}\rangle_L+
\langle M\times_L\vec{c},
\vec{c}^{*}\rangle_L
\Big).
\end{aligned}
\]
By the alternating property
of the Lorentzian scalar
triple product,
\[
\langle\vec{c},
M\times_L\vec{c}^{*}\rangle_L
=
-\langle M\times_L\vec{c},
\vec{c}^{*}\rangle_L.
\]
Therefore, we acquire
\[
\langle\vec{C},
\vec{C}^{*}\rangle_L
=
\langle\vec{c},
\vec{c}^{*}\rangle_L=0.
\]
It follows that
\[
\begin{aligned}
\langle\Gamma,\Gamma\rangle_L
&=
\langle\vec{C},\vec{C}\rangle_L
+2\varepsilon^{*}
\langle\vec{C},
\vec{C}^{*}\rangle_L=\sigma.
\end{aligned}
\]
Similarly, we get
\[
\begin{aligned}
\langle\Gamma_1,
\Gamma_1\rangle_L
&=
\langle\vec{C}^{*},
\vec{C}^{*}\rangle_L
+2\varepsilon^{*}
\langle\vec{C}^{*},
\vec{C}\rangle_L=\tau.
\end{aligned}
\]
Hence, we have
\[
\Gamma\in
\begin{cases}
\hat{\mathbb{S}}_1^2,
&\sigma=1,\\
\hat{\mathbb{H}}_0^2,
&\sigma=-1,
\end{cases}
\]
and
\[
\Gamma_1\in
\begin{cases}
\hat{\mathbb{S}}_1^2,
&\tau=1,\\
\hat{\mathbb{H}}_0^2,
&\tau=-1.
\end{cases}
\]
This completes the proof.
\end{proof}
\begin{rem}
Proposition \ref{412} should not be read as a converse of Theorem \ref{42}. If the curve
$\Gamma=\vec C+\varepsilon^{*}\vec C^{*}$ constructed there is substituted into
Definition \ref{41}, the real base curve of the resulting surface is
$\vec c\times_L\vec c^{*}$, which in general differs from $M$. Thus $\Gamma$ only
encodes the two families of ruling lines via the E.~Study map; it does not
reproduce the given surfaces through the construction of Definition~\ref{41}.
\end{rem}
\begin{prop}
Let $\Gamma(t)=\vec{C}(t)
+\varepsilon^{*}\vec{C}^{*}(t)$
be a smooth hyper-dual Lorentzian curve
on $\hat{\mathbb{S}}_{1}^{2}$
(or $\hat{\mathbb{H}}_{0}^{2}$).
Define $\beta_0(t)
=\vec{c}_{0}(t)\times_L\vec{c}_{2}(t)$
and $\beta_1(t)
=\vec{c}_{1}(t)\times_L\vec{c}_{2}(t)
+\vec{c}_{0}(t)\times_L\vec{c}_{3}(t).$
Assume that $\vec{c}_{1}(t)$
is non-null and $\beta_1(t)-\beta_0(t)
\in
\operatorname{span}
\{\vec{c}_{0}(t),\vec{c}_{1}(t)\}$
for all $t$ in the parameter interval. Then, for every smooth function
$h(t)$, the real and dual components
of the associated non-null ruled
surface determine two
ruled parametrizations in
$\mathbb{R}_{1}^{3}$ having a
common base curve and mutually
orthogonal non-null ruling directions. \\
Wherever both parametrizations are
regular immersions with
non-degenerate induced metrics,
they define a pair of non-degenerate
non-null ruled surfaces.
\end{prop}
\begin{proof}
Consider the associated dual ruled surface
\[
\phi(t,V)
=\vec{C}(t)\times_L\vec{C}^{*}(t)
+V\vec{C}(t).
\]
Using $\varepsilon^2=0$, we obtain
\[
\begin{aligned}
\phi(t,V)
={}&
\beta_0(t)+v\vec{c}_{0}(t)\\
+&\varepsilon\Big(
\beta_1(t)+v\vec{c}_{1}(t)
+v^{*}\vec{c}_{0}(t)
\Big).
\end{aligned}
\]
Fixing $v^{*}=h(t)$, the
corresponding real and dual
components are
\[
I(t,v)
=\beta_0(t)+v\vec{c}_{0}(t)
\]
and
\[
II_h(t,v)
=\beta_1(t)+h(t)\vec{c}_{0}(t)
+v\vec{c}_{1}(t).
\]
Since $\Gamma(t)$ is a unit
hyper-dual vector, $\langle\vec{C},\vec{C}\rangle_L
=\sigma$ and $\sigma\in\{1,-1\}.$
Comparing dual coefficients yields $\langle\vec{c}_{0},\vec{c}_{0}\rangle_L
=\sigma$
and $\langle\vec{c}_{0},\vec{c}_{1}\rangle_L
=0.$
Thus, the ruling directions are
Lorentzian orthogonal. Now set
\[
D_h(t)
=\beta_1(t)+h(t)\vec{c}_{0}(t)
-\beta_0(t).
\]
By the additional span condition,
\[
D_h(t)\in
\operatorname{span}
\{\vec{c}_{0}(t),\vec{c}_{1}(t)\}.
\]
Since $\vec{c}_{0}$ and
$\vec{c}_{1}$ are non-null
and orthogonal, there exist
unique smooth functions
$f(t)$ and $g(t)$ such that
\[
D_h(t)
=f(t)\vec{c}_{0}(t)
-g(t)\vec{c}_{1}(t).
\]

These functions are
\[
f(t)=
\frac{
\langle D_h(t),\vec{c}_{0}(t)\rangle_L
}{
\langle\vec{c}_{0}(t),
\vec{c}_{0}(t)\rangle_L
}
\]
and
\[
g(t)=-
\frac{
\langle D_h(t),\vec{c}_{1}(t)\rangle_L
}{
\langle\vec{c}_{1}(t),
\vec{c}_{1}(t)\rangle_L
}.
\]
Consequently,
\[
\begin{aligned}
\beta_0(t)+f(t)\vec{c}_{0}(t)
={}&\beta_1(t)+h(t)\vec{c}_{0}(t)
+g(t)\vec{c}_{1}(t).
\end{aligned}
\]
Define the common base curve by
\[
M_h(t)=\beta_0(t)
+f(t)\vec{c}_{0}(t).
\]
The two associated ruled
parametrizations can therefore
be written as
\[
\bar{\phi}_h(t,v)
=M_h(t)+v\vec{c}_{0}(t)
\]
and
\[
\bar{\phi}_{1,h}(t,w)
=M_h(t)+w\vec{c}_{1}(t).
\]
They are obtained from
$I$ and $II_h$, respectively,
by shifts of the ruling
parameters. Furthermore,
\[
\langle\vec{c}_{0}(t),
\vec{c}_{1}(t)\rangle_L=0.
\]
Hence, the two parametrizations
have a common base curve and
mutually orthogonal non-null
ruling directions. The conclusion concerning
non-degenerate non-null ruled surfaces
follows at points where both
parametrizations are regular
immersions and their induced
metrics are non-degenerate.
\end{proof}

\begin{rem}
The additional assumptions are
essential. In particular, the
hyper-dual unit condition alone
does not guarantee that
$\vec{c}_{1}$ is non-null or
that the two associated non-null ruled
surfaces admit a common base curve. Although the construction allows
an arbitrary smooth function
$h(t)$, different choices of
$h$ need not produce geometrically
distinct non-null ruled surface pairs.
\end{rem}

\begin{example}
Let
\[
M(t)=(0,\sinh t,\cosh t)
\]
be a spacelike curve in
$\mathbb{R}_1^3$. Its unit tangent and unit timelike
principal normal vectors are
\[
\begin{aligned}
\vec{t}(t)&=(0,\cosh t,\sinh t),\\
\vec{n}(t)&=(0,\sinh t,\cosh t).
\end{aligned}
\]

Let $\vec{e}_{1}=(1,0,0)$ and consider
the dual Lorentzian vectors
\[
\begin{aligned}
\vec{C}(t)
&=\vec{n}(t)+\varepsilon\vec{e}_1,\\
\vec{C}^{*}(t)
&=\vec{t}(t).
\end{aligned}
\]
Explicitly, we can write
\[
\begin{aligned}
\vec{C}(t)
&=(\varepsilon,\sinh t,\cosh t),\\
\vec{C}^{*}(t)
&=(0,\cosh t,\sinh t).
\end{aligned}
\]
Define the hyper-dual curve
\[
\Gamma(t)
=
\vec{C}(t)+
\varepsilon^{*}\vec{C}^{*}(t).
\]
We have $\langle\vec{C},\vec{C}\rangle_L=-1$ and $\langle\vec{C}^{*},
\vec{C}^{*}\rangle_L=1.$
Consequently, we have $\Gamma(t)\in
U\hat{\mathbb{H}}_0^2.$
In terms of the component vectors, we write
\[
\begin{aligned}
\vec{c}_0(t)&=\vec{n}(t),\\
\vec{c}_1(t)&=\vec{e}_1,\\
\vec{c}_2(t)&=\vec{t}(t),\\
\vec{c}_3(t)&=\vec{0}.
\end{aligned}
\]
The corresponding base curves
of the real and dual components
are
\[
\begin{aligned}
\beta_0(t)
&=\vec{c}_0(t)\times_L
\vec{c}_2(t)\\
&=\vec{n}(t)\times_L
\vec{t}(t)\\
&=-\vec{e}_1
\end{aligned}
\]
and
\[
\begin{aligned}
\beta_1(t)
&=\vec{c}_1(t)\times_L
\vec{c}_2(t)
+\vec{c}_0(t)\times_L
\vec{c}_3(t)\\
&=\vec{e}_1\times_L
\vec{t}(t)\\
&=-\vec{n}(t).
\end{aligned}
\]
Therefore, we write
\[
\begin{aligned}
\beta_1(t)-\beta_0(t)
&=-\vec{n}(t)+\vec{e}_1\\
&=-\vec{c}_0(t)+\vec{c}_1(t).
\end{aligned}
\]
Hence, we have
\[
\beta_1(t)-\beta_0(t)
\in
\operatorname{span}
\{\vec{c}_0(t),\vec{c}_1(t)\}.
\]
Thus, the additional assumption
of the preceding proposition
is satisfied. Now let $a>1$ be a real constant
and choose the dual parameter $V=v+\varepsilon a.$
The associated dual ruled
parametrization is
\[
\phi(t,V)
=
\vec{C}(t)\times_L
\vec{C}^{*}(t)
+V\vec{C}(t).
\]
Expanding this expression yields
\[
\begin{aligned}
\phi(t,V)
={}&-\vec{e}_1+v\vec{n}(t)\\
&+\varepsilon\Big(
(a-1)\vec{n}(t)
+v\vec{e}_1
\Big).
\end{aligned}
\]
Thus, its real and dual
components are
\[
I(t,v)
=-\vec{e}_1+v\vec{n}(t)
\]
and
\[
II_a(t,v)
=(a-1)\vec{n}(t)
+v\vec{e}_1.
\]
Since
\[
\beta_1-\beta_0
=-\vec{c}_0+\vec{c}_1,
\]
the functions introduced
in the preceding proposition
are
\[
f(t)=a-1,
\qquad
g(t)=-1.
\]
Therefore, the common base curve
is
\[
\begin{aligned}
M_a(t)
&=\beta_0(t)
+f(t)\vec{c}_0(t)\\
&=-\vec{e}_1+(a-1)\vec{n}(t).
\end{aligned}
\]
Explicitly, we have
\[
M_a(t)=
\left(
-1,\,
(a-1)\sinh t,\,
(a-1)\cosh t
\right).
\]
After suitable shifts of the
ruling parameters, the associated
ruled surfaces are
\[
\begin{aligned}
\bar{\phi}_a(t,v)
&=M_a(t)+v\vec{n}(t)\\
&=
\left(
-1,\,
(a-1+v)\sinh t,\,
(a-1+v)\cosh t
\right)
\end{aligned}
\]
and
\[
\begin{aligned}
\bar{\phi}_{1,a}(t,w)
&=M_a(t)+w\vec{e}_1\\
&=
\left(
w-1,\,
(a-1)\sinh t,\,
(a-1)\cosh t
\right).
\end{aligned}
\]
Clearly, both surfaces share
the base curve $M_a(t)$. Their ruling directions
are $\vec{n}(t)$ and
$\vec{e}_1$, respectively. Moreover, $\langle\vec{n}(t),
\vec{e}_1\rangle_L=0.$
Thus, the ruling directions
are mutually orthogonal. We now determine the causal
characters of these surfaces. For $\bar{\phi}_a$, we obtain
\[
\begin{aligned}
(\bar{\phi}_a)_t
&=(a-1+v)\vec{t}(t),\\
(\bar{\phi}_a)_v
&=\vec{n}(t).
\end{aligned}
\]
Hence, the first fundamental
form coefficients are
\[
E=(a-1+v)^2,
\qquad
F=0,
\qquad
G=-1.
\]
Therefore, we compute
\[
EG-F^2=-(a-1+v)^2<0
\]
whenever $v\neq1-a$. Consequently, $\bar{\phi}_a$
is a regular timelike ruled
surface on this domain. For $\bar{\phi}_{1,a}$,
we have
\[
\begin{aligned}
(\bar{\phi}_{1,a})_t
&=(a-1)\vec{t}(t),\\
(\bar{\phi}_{1,a})_w
&=\vec{e}_1.
\end{aligned}
\]
Its first fundamental form
coefficients are
\[
E_1=(a-1)^2,
\qquad
F_1=0,
\qquad
G_1=1.
\]
Thus, we calculate
\[
E_1G_1-F_1^2=(a-1)^2>0.
\]
Since $a>1$,
$\bar{\phi}_{1,a}$ is a
regular spacelike ruled
surface. Finally, for distinct
values of $a>1$, the surfaces
$\bar{\phi}_{1,a}$ satisfy
\[
x_2^2-x_3^2=-(a-1)^2.
\]
Consequently, distinct values
of $a$ determine geometrically
distinct spacelike ruled
surfaces. We therefore obtain infinitely
many pairs of non-degenerate
ruled surfaces in
$\mathbb{R}_1^3$ with a common
base curve and mutually
orthogonal ruling directions.
	\begin{figure}[!htp]
	\centering
	\includegraphics[width=8cm,height=4cm]{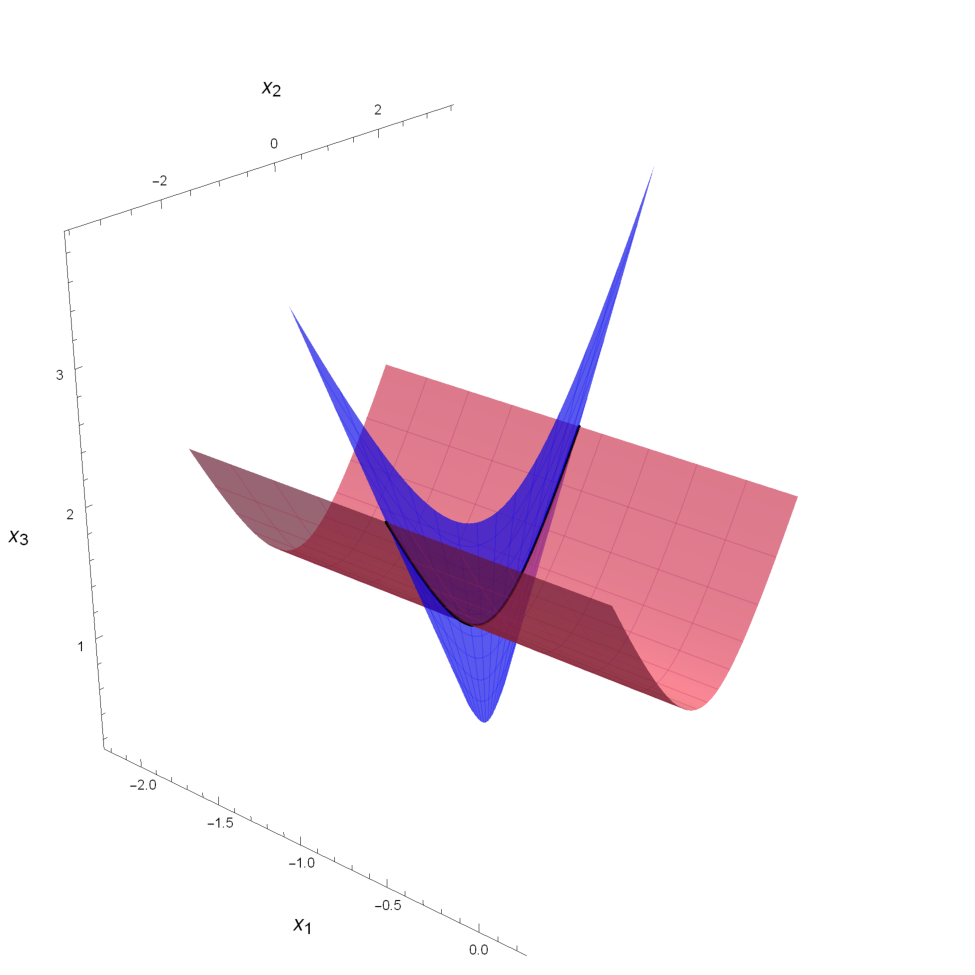}
	\caption{A pair of non-degenerate ruled surfaces in Minkowski 3-space with a common base curve and mutually orthogonal ruling directions. The blue surface is timelike, the pink surface is spacelike, and the black curve represents their common base curve ($M_a(t)$), for ($a=2$).}
	\label{fig: new}
\end{figure}
\end{example}

\begin{example}
Let $\alpha:I\rightarrow\mathbb{R}_1^3$ be
the unit-speed spacelike curve
\[
\alpha(t)=(\cos t,\sin t,0).
\]
Its Frenet vectors are
\[
\begin{aligned}
\vec{t}(t)&=(-\sin t,\cos t,0),\\
\vec{n}(t)&=(-\cos t,-\sin t,0),\\
\vec{b}(t)&=(0,0,-1).
\end{aligned}
\]
Consider the dual Lorentzian vectors
\[
\begin{aligned}
\vec{C}(t)
&=\vec{t}(t)+\varepsilon\vec{b}(t),\\
\vec{C}^{*}(t)
&=\sqrt{2}\vec{n}(t)+\vec{b}(t)+\varepsilon\left(
\vec{t}(t)
+\frac{1}{\sqrt{2}}\vec{n}(t)
+\vec{b}(t)
\right).
\end{aligned}
\]
Define the hyper-dual Lorentzian curve
\[
\Gamma(t)
=
\vec{C}(t)
+\varepsilon^{*}\vec{C}^{*}(t).
\]
First, we verify that
$\Gamma(t)\in U\hat{\mathbb{S}}_1^2$. Indeed,
\[
\langle\vec{C},\vec{C}\rangle_L
=
\langle\vec{t},\vec{t}\rangle_L
+
2\varepsilon
\langle\vec{t},\vec{b}\rangle_L
=1.
\]
Moreover, we have $\langle\vec{C}^{*},
\vec{C}^{*}\rangle_L
=1$ and $\langle\vec{C},\vec{C}^{*}\rangle_L=0.$
Consequently,
\[
\langle\Gamma,\Gamma\rangle_L=1,
\qquad
\langle\vec{C}^{*},
\vec{C}^{*}\rangle_L=1,
\]
and hence
\[
\Gamma(t)\in U\hat{\mathbb{S}}_1^2.
\]
The corresponding component vectors are
\[
\begin{aligned}
\vec{c}_0&=\vec{t},\\
\vec{c}_1&=\vec{b},\\
\vec{c}_2&=\sqrt{2}\vec{n}+\vec{b},\\
\vec{c}_3&=\vec{t}
+\frac{1}{\sqrt{2}}\vec{n}
+\vec{b}.
\end{aligned}
\]
In particular, $\langle\vec{c}_0,\vec{c}_1\rangle_L=0$ and $\langle\vec{c}_1,\vec{c}_1\rangle_L=-1.$
Now consider the associated dual
ruled parametrization
\[
\phi(t,V)
=
\vec{C}(t)\times_L\vec{C}^{*}(t)
+V\vec{C}(t),
\]
Its real and dual base curves are
\[
\beta_0(t)
=
\vec{c}_0(t)\times_L\vec{c}_2(t)
\]
and
\[
\beta_1(t)
=
\vec{c}_1(t)\times_L\vec{c}_2(t)
+
\vec{c}_0(t)\times_L\vec{c}_3(t).
\]
Using the Lorentzian vector product,
we obtain
\[
\vec{t}\times_L\vec{n}=\vec{b},
\qquad
\vec{t}\times_L\vec{b}=\vec{n},
\qquad
\vec{b}\times_L\vec{n}=\vec{t}.
\]
Therefore, we acquire
\[
\begin{aligned}
\beta_0(t)
&=\vec{t}\times_L
(\sqrt{2}\vec{n}+\vec{b})\\
&=\sqrt{2}\vec{b}+\vec{n}
\end{aligned}
\]
and
\[
\begin{aligned}
\beta_1(t)
&=\vec{b}\times_L
(\sqrt{2}\vec{n}+\vec{b})\\
&\quad+
\vec{t}\times_L
\left(
\vec{t}
+\frac{1}{\sqrt{2}}\vec{n}
+\vec{b}
\right)\\
&=\sqrt{2}\vec{t}
+\frac{1}{\sqrt{2}}\vec{b}
+\vec{n}.
\end{aligned}
\]
Consequently, we write
\[
\begin{aligned}
\beta_1(t)-\beta_0(t)
&=\sqrt{2}\vec{t}
-\frac{1}{\sqrt{2}}\vec{b}\\
&=\sqrt{2}\vec{c}_0
-\frac{1}{\sqrt{2}}\vec{c}_1.
\end{aligned}
\]
Thus, we get
\[
\beta_1(t)-\beta_0(t)
\in
\operatorname{span}
\{\vec{c}_0(t),\vec{c}_1(t)\}.
\]
Hence, the additional assumption
of the preceding proposition
is satisfied. Taking $v^{*}=0$, the functions
determining a common base curve are $f(t)=\sqrt{2}$ and $g(t)=\frac{1}{\sqrt{2}}.$
We therefore define
\[
\begin{aligned}
M(t)
&=\beta_0(t)+f(t)\vec{c}_0(t)\\
&=\sqrt{2}\vec{b}(t)
+\vec{n}(t)
+\sqrt{2}\vec{t}(t).
\end{aligned}
\]
Explicitly, we write
\[
M(t)=
\left(
-\sqrt{2}\sin t-\cos t,\,
\sqrt{2}\cos t-\sin t,\,
-\sqrt{2}
\right).
\]
Moreover,
\[
M(t)
=
\beta_1(t)+g(t)\vec{c}_1(t).
\]
Thus, the two associated ruled surfaces
with the common base curve $M(t)$ are
\[
\begin{aligned}
\bar{\phi}(t,v)
&=M(t)+v\vec{t}(t),\\
\bar{\phi}_1(t,w)
&=M(t)+w\vec{b}(t).
\end{aligned}
\]
Their explicit parametrizations are
\[
\begin{aligned}
\bar{\phi}(t,v)
=\big(&
-(\sqrt{2}+v)\sin t-\cos t,\\
&
(\sqrt{2}+v)\cos t-\sin t,\\
&
-\sqrt{2}
\big)
\end{aligned}
\]
and
\[
\begin{aligned}
\bar{\phi}_1(t,w)
=\big(&
-\sqrt{2}\sin t-\cos t,\\
&
\sqrt{2}\cos t-\sin t,\\
&
-\sqrt{2}-w
\big).
\end{aligned}
\]
The ruling directions are
$\vec{t}(t)$ and $\vec{b}(t)$,
respectively. Since
\[
\langle\vec{t}(t),
\vec{b}(t)\rangle_L=0,
\]
the two surfaces have mutually
orthogonal ruling directions. We now examine their causal characters.
For the first surface,
\[
\begin{aligned}
\bar{\phi}_t
&=(\sqrt{2}+v)\vec{n}-\vec{t},\\
\bar{\phi}_v
&=\vec{t}.
\end{aligned}
\]
Therefore, its first fundamental
form coefficients are
\[
E=(\sqrt{2}+v)^2+1,
\qquad
F=-1,
\qquad
G=1.
\]
Consequently, we compute
\[
EG-F^2=(\sqrt{2}+v)^2>0
\]
for $v\neq-\sqrt{2}$. Hence, $\bar{\phi}$ is a regular
spacelike ruled surface on any
domain excluding $v=-\sqrt{2}$. For the second surface,
\[
\begin{aligned}
(\bar{\phi}_1)_t
&=\sqrt{2}\vec{n}-\vec{t},\\
(\bar{\phi}_1)_w
&=\vec{b}.
\end{aligned}
\]
Thus, we compute
\[
E_1=3,
\qquad
F_1=0,
\qquad
G_1=-1,
\]
and
\[
E_1G_1-F_1^2=-3<0.
\]
Therefore, $\bar{\phi}_1$
is a regular timelike ruled surface. Furthermore, the first surface lies
in the spacelike plane
\[
x_3=-\sqrt{2},
\]
whereas the second surface lies
on the timelike circular cylinder
\[
x_1^2+x_2^2=3.
\]
Their common base curve $M(t)$
is contained in the intersection
of these two surfaces. Thus, the hyper-dual curve
$\Gamma(t)\in U\hat{\mathbb{S}}_1^2$
determines a pair of non-degenerate
ruled surfaces in $\mathbb{R}_1^3$
with a common base curve and
mutually orthogonal non-null
ruling directions.

\begin{figure}[!htp]
    \centering
    \includegraphics[width=9cm]{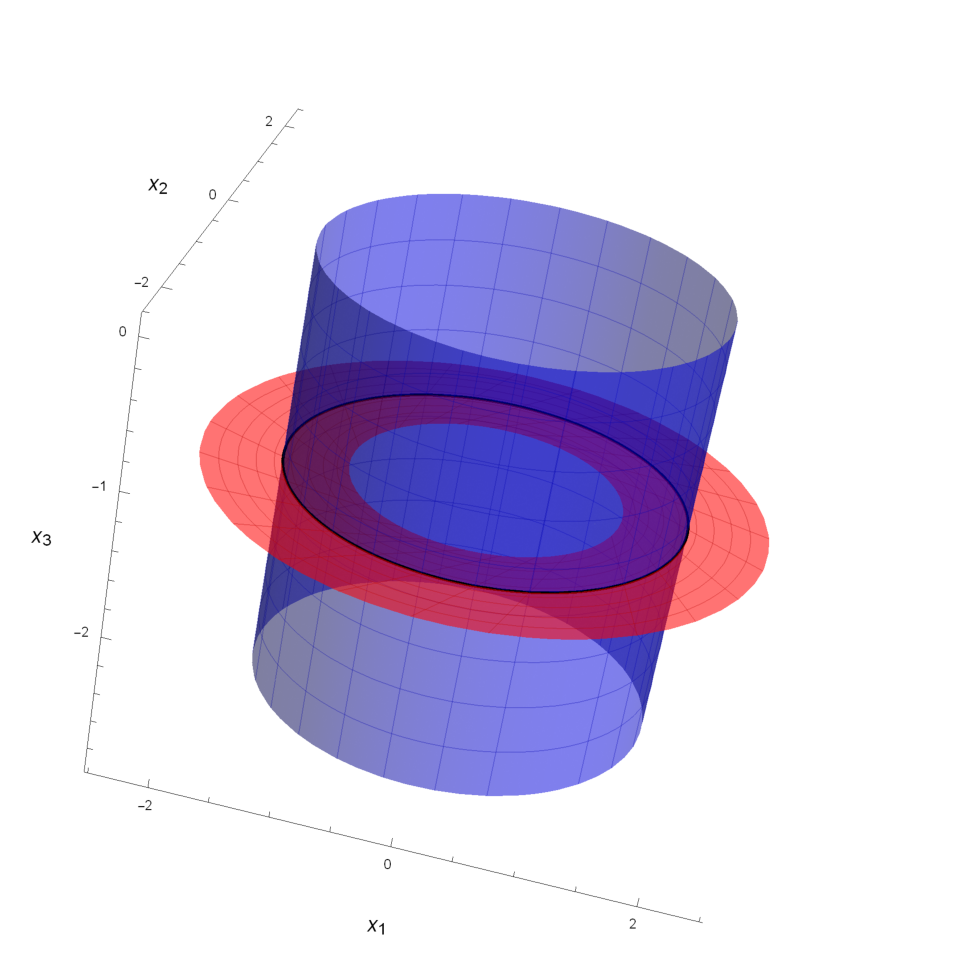}
    \caption{A spacelike ruled surface (red)
    and a timelike ruled surface (blue)
    in Minkowski $3$-space with a common
    base curve (black) and mutually
    orthogonal ruling directions.}
    \label{fig:new1}
\end{figure}

\end{example}
\section{Conclusion}

In this paper, a geometric framework relating pseudo hyper-dual tangent sphere bundles to the theory of non-null ruled surfaces is established in the hyper-dual Lorentzian setting. Extending the classical E. Study correspondence to unit hyper-dual vectors, we obtain a natural representation of non-null ruled surfaces generated by hyper-dual Lorentzian curves. This correspondence provides a geometric interpretation of higher-order dual structures in terms of ruled surface theory and unifies the underlying relationships between hyper-dual curves and tangent sphere bundles.

Within this framework, non-null ruled surfaces are systematically classified. Moreover, necessary and sufficient developability conditions are derived and interpreted geometrically through hyper-dual Lorentzian curves and their arc-length parameters. Additional structural properties, common base curves, and orthogonal ruling directions, are established. The theoretical results are supported by explicit examples and graphical illustrations, confirming the validity of the proposed constructions.

The results presented here extend the scope of the classical E. Study map from the dual Lorentzian setting to pseudo hyper-dual tangent sphere bundles and provide a unified approach to the geometric analysis of non-null ruled surfaces associated with higher-order dual structures. Beyond their theoretical significance, these results offer a foundation for further investigations on hyper-dual differential geometry, tangent bundle theory, and the geometry of ruled surfaces in semi-Riemannian spaces. Future research may consider analogous constructions for null ruled surfaces, higher-dimensional semi-Riemannian manifolds, and alternative generalized dual algebraic structures.
\section*{}

\end{document}